\documentclass[12pt]{amsart}
\usepackage{amsfonts, amssymb, amsmath, amsthm, euscript}
\swapnumbers
\theoremstyle{plain}
\newtheorem{thm}{Theorem}[section]
\newtheorem{lem}[thm]{Lemma}
\newtheorem{prop}[thm]{Proposition}

\theoremstyle{definition}
\newtheorem{solution}[thm]{Solution}
\newtheorem{nonsingular}[thm]{Nonsingular solutions}

\newtheorem*{Ack}{Acknowledgement}

\theoremstyle{remark}

\def\K{\mathbb{K}}
\def\A{\mathbb{A}}
\def\e{\mathbf{e}}

\def\f{\frac}

\begin{document}
\thispagestyle{empty}

\title[Irreducible representations by matrix solutions]{Irreducible representations of the quantum Weyl algebra at roots of unity given by matrices }

\author{Blaise Heider and Linhong Wang}

\address{Department of Mathematics\\ Southeastern Louisiana University\\ Hammond, LA 70402}

\email{Blaise.Heider@selu.edu, lwang@selu.edu}

\thanks{The first author was supported by an undergraduate research grant from CST at Southeastern Louisiana University. Research of the second author supported by the Louisiana Board of Regents [LEQSF(2012-15)-RD-A-20].}



\begin{abstract}
To describe the representation theory of the quantum Weyl algebra at an $l$th primitive root $\gamma$ of unity, Boyette, Leyk, Plunkett, Sipe, and Talley found all nonsingular irreducible matrix solutions to the equation $yx-\gamma xy=1$, assuming $yx\neq xy$. In this note, we complete their result by finding and classifying, up to equivalence, all irreducible matrix solutions $(X, Y)$, where $X$ is singular.
\end{abstract}

\maketitle

\section{Introduction}
\thispagestyle{empty}
Let $\K$ be an algebraically closed field, $\gamma$ a nonzero scalar in $\K$. The irreducible representations of the quantized Weyl algebra $\A=\K\{x,\,y\}/\langle yx-\gamma xy -1\rangle$ has been constructed and classified in \cite{DGO} and \cite{Jor}. It is well known that, when $\gamma$ is a root of unity, any irreducible representation of $\A$ is finite dimensional. In \cite{BLTPS}, Boyette, Leyk, Plunkett, Sipe, and Talley present a linear algebra method to prove the Drozd-Guzner-Ovsienko result in the case when $\gamma$ is a primitive root of unity. Let $V$ be an $\A$-module and $\rho: \A \rightarrow \text{End}(V) \hookrightarrow M_n(\K)$ be a representation of $\A$. Then $\rho$ is irreducible, i.e., $V$ is simple, if and only if $\rho(\A)=\text{M}_n(\K)$. Their approach then is to find explicitly, up to equivalence, all irreducible matrix solutions to the equation $yx-\gamma xy =1$ with $yx \neq xy$. Two $n\times n$ matrices $X$ and $Y$ form a \emph{matrix solution} if $YX-\gamma XY=I$, the identity matrix. A solution $(X,Y)$ is \emph{irreducible} if every matrix in $M_n(\K)$ can be written as a noncommutative polynomial in $X$ and $Y$ over $\K$, assuming the zero power of a matrix is the identity matrix. Two solutions $(X, Y)$ and $(M, N)$ are \emph{equivalent} if there is a nonsingular matrix $Q$ such that $QXQ^{-1}=M$ and $QYQ^{-1}=N$.

Throughout, $\gamma$ is an $l$th primitive root of unity for some integer $l \geq 2$. Unless specified otherwise, any \emph{solution} is a matrix solution to the equation $yx-\gamma xy =1$.

\begin{nonsingular}\label{nonsingular}
Suppose $(X, Y)$ is an irreducible solution and $U=YX-XY \neq 0$. The following are proved in \cite{BLTPS}.

(i) $YX^l=X^lY$ and $Y^lX=XY^l$, and so $X^l$ and $Y^l$ are scalar matrices. It follows that any irreducible matrix solutions is at most $l\times l$.

(ii) $UX=\gamma XU$, $YU=\gamma UY$, and $U$ is nonsingular. If $X$ has one nonzero eigenvalue, say $\lambda$, then $X$ has at least $l$ distinct eigenvalues, $\gamma\lambda, \, \gamma^2\lambda, \, \ldots, \gamma^l\lambda$. In particular, $(X, Y)$ is at least $l\times l$.

(iii) When $X$ has at least one nonzero eigenvalue, all irreducible solutions are $l\times l$ and found explicitly in \cite{BLTPS} as follows.
{\small
\[ X_{\lambda}=\lambda \left(%
 \begin{array}{cccc}
\gamma & 0 & \ldots &0  \\
 0 & \gamma^2 &\ldots &0  \\
 \ &\ & \ & \  \\
 \vdots & \vdots & \ddots &\vdots   \\
 \ &\ & \ & \  \\
 0 & 0 & \ldots & \gamma^l \\
\end{array}%
\right)\;
Y_{\lambda\, b's} =\left(%
\begin{array}{cccccc}
 \frac{1}{(1-\gamma)\gamma \lambda} & b_1 &0 &\ldots   &0 &0  \\
  0 & \frac{1}{(1-\gamma)\gamma^2 \lambda} &b_2&\ldots   &0 &0  \\
  \ & \ & \ &\ & \ &\  \\
   \vdots & \vdots  & \vdots &\ddots  &\vdots & \vdots \\
 \ & \ & \ &\ & \ &\ \\
 0& 0 & 0 &\ldots & \frac{1}{(1-\gamma)\gamma^{l-1} \lambda}& b_{l-1} \\
b_l &0 &0 & \ldots & 0 &   \frac{1}{(1-\gamma)\gamma^l \lambda} \\
\end{array}%
\right)
\]
} where $\lambda, b$'s are nonzero scalars in $\K$. The two matrices in each of these solutions are both nonsingular. These solutions are corresponding to the $x$-, $y$-torsion-free simple $\A$-modules. (cf. \cite[Summary 3.7]{Jor})
\end{nonsingular}

However, for an irreducible solution $(X, Y)$, it is not always true that $X$ has nonzero eigenvalues. For example, when $l=2$, $\gamma=-1$, the matrices \[X=\left(
\begin{array}{cc}
0 & 1 \\
0 & 0 \\
\end{array}
\right)\quad \text{and} \quad
Y=\left(
\begin{array}{cc}
0 & 0 \\
1 & 0 \\
\end{array}
\right)\] form an irreducible solution to the equation $yx+xy=1$. In fact, this solution gives the only simple $A$-module, where $\A=\K \{x, y\}/\langle
yx+xy=1\rangle$, that is both $x$- and $y$-torsion. (cf. \cite[Example 4.1]{Jor})  Hence, the remark (ii) in \cite{BLTPS} and the assertion in \cite[Proposition 2]{BLTPS} are incorrect.

In this note, we complete \cite[Proposition 2]{BLTPS} by finding, up to equivalence, all irreducible matrix solutions $(X,Y)$, where $X$ only has zero eigenvalue. Moreover, the two types of irreducible matrix solutions, with $X$ nonsingular or singular, are classified up to equivalence. These irreducible solutions give explicitly all irreducible representations of the quantum Weyl algebra at the root $\gamma$.

\begin{Ack}
This problem arose in conversations between the second author and her Ph.~D. advisor, E.~Letzter. The authors are thankful for his generosity. The authors are also grateful to the referee for the comments that lead to Lemma \ref{reduce-singular}.
\end{Ack}

\section{irreducible matrix solutions with $X$ singular}
Direct computation shows that the following is a solution.

\begin{solution}\label{solution}
{\small
\[X= \left(%
\begin{array}{ccccc}
 0 & 1 & \ &\ & \  \\
 \ & 0 & 1 &\ & \  \\
 \ & \ & \ &\ddots &   \\
 \ & \ & \ & 0 &  1 \\
 \ & \ & \ & \ & 0 \\
\end{array}%
\right)_{l\times l}
Y=Y_{\beta} =\left(%
\begin{array}{cccccc}
 0 & 0 &\ldots   &0 &0 &\beta  \\
  \sum_{i=0}^{l-2} \gamma^{i} & 0 &\ldots   &0 &0 &0  \\
   0 & \sum_{i=0}^{l-3} \gamma^{i} &\ldots  &0  &0 &0 \\

   \vdots & \vdots &\ddots &\vdots & \vdots &\vdots \\

 0& 0 &\ldots  & 1+\gamma &0& 0 \\
 0 &0 &\ldots & 0 & 1&  0 \\
\end{array}%
\right)\]
}
where $\beta$ is a scalar in $\K$. The following on the matrices $X$ and $Y$ are either standard or straightforward.

(i) Note that $X$ is the well-known upper shift matrix, and $X^l=0$. Fix a positive integer $u\leq l-1$. Then $X^u$ is a matrix with $u$th superdiagonal line all ones and zeroes elsewhere. Premultiplying a matrix $A$ by $X^u$ results in a matrix whose last $u$ rows are all zeroes and the first $l-u$ rows are the last $l-u$ rows of the matrix $A$.

(ii) Note that each $Y_{k, k-1}$ along the subdiagonal line of $Y$ is the only nonzero entry on the $k$th row of $Y$. Fix a positive integer $v\leq l-1$. It follows that
\[
(Y^v)_{l,k}
=\left\{\begin{array}{ll}
0 & k\neq l-v\\
\prod_{j=l-v+1}^{l} \sum_{i=0}^{l-j} \gamma^{i} & k=l-v
\end{array}\right.
\]
That is, the only nonzero entry on the $l$th row of the matrix $Y^v$ is $(Y^v)_{l, l-v}$.
\end{solution}

\begin{prop}
The solution $(X,Y)$ in \emph{(\ref{solution})} is irreducible.
\end{prop}

\begin{proof}
Let $H=\text{Span}_{\K}\{X^iY^j; \quad 0\leq i, j\leq l-1\}$. For any positive integers $m,\,  n \leq l$, we will show that the elementary matrices $\e_{mn}$ are in $H$.

For fixed integers $m,\,  n \leq l$, consider the matrix $X^{l-m}Y^{l-n}$. By (\ref{solution} i), the last nonzero row of $X^{l-m}Y^{l-n}$ is the $m$th row, which is the same as the $l$th row of the matrix $Y^{l-n}$. By (\ref{solution} ii), the only nonzero entry in the $l$th row of $Y^{l-n}$ is
\[(Y^{l-n})_{l, n}=\prod_{j=n+1}^{l} \sum_{i=0}^{l-j} \gamma^{i}
\]
Hence,
\[X^{l-m}Y^{l-n}=\Big(\prod_{j=n+1}^{l} \sum_{i=0}^{l-j} \gamma^{i}\Big) \e_{mn} + \Big(\sum_{s=1}^{m-1} \sum_{t=1}^l (X^{l-m}Y^{l-n})_{s,t}\Big) \e_{st}\]
When $m=1$, we have
\[X^{l-1}Y^{l-n}=\Big(\prod_{j=n+1}^{l} \sum_{i=0}^{l-j} \gamma^{i}\Big) \e_{1n}\]
for $n=1, \ldots, l$. Then, by induction on $m$, it shows that any $\e_{mn} \in H$.
\end{proof}

\begin{prop}\label{singular}
Suppose $(A, B)$ with $BA \neq AB$ is an irreducible solution. If the only eigenvalue of $A$ is $0$, then $(A, B)$ is equivalent to a solution as follows.
{\small
\[X = \left(%
\begin{array}{ccccc}
 0 & 1 & \ &\ & \  \\
 \ & 0 & 1 &\ & \  \\
 \ & \ & \ &\ddots &   \\
 \ & \ & \ & 0 &  1 \\
 \ & \ & \ & \ & 0 \\
\end{array}%
\right)_{l\times l}
Y_{\alpha'\text{s}} =\left(%
\begin{array}{cccccc}
 \gamma^{l-1}\alpha_l & \gamma^{l-2}\alpha_{l-1} &\ldots   &\gamma^2\alpha_3 &\gamma\alpha_2 &\alpha_1  \\
  \sum_{i=0}^{l-2} \gamma^{i} & \gamma^{l-2}\alpha_{l} &\ldots   &\gamma^2\alpha_4 &\gamma\alpha_3 &\alpha_2  \\
   0 & \sum_{i=0}^{l-3} \gamma^{i} &\ldots  &\gamma^2\alpha_5  &\gamma\alpha_4 &\alpha_3 \\

   \vdots & \vdots &\ddots  &\vdots & \vdots &\vdots \\

 0& 0 &\ldots & 1+\gamma &\gamma \alpha_l& \alpha_{l-1} \\
 0 &0 &\ldots & 0 & 1&  \alpha_l \\
\end{array}%
\right)\]}
where $\alpha_i$ are scalars in $\K$ for $i=1,2,\ldots, l$.
\end{prop}

\begin{proof}
Suppose that $(A, B)$ with $BA \neq AB$ is an $n\times n$ irreducible solution for some positive integer $n$ and that the only eigenvalue of $A$ is $0$. Let $C=QAQ^{-1}$ be the Jordan Canonical form for $A$, where $Q$ is an $n\times n$ matrix. Set $D=QBQ^{-1}$. Then $(C,\,D)$ is also an irreducible solution.

Suppose the Jordan blocks $J_i$ in $C$ are $m_i\times m_i$ for $i=1, \ldots, k$. Partition the matrix $D$ into blocks $(D_{ij})$, where $D_{ij}$ are $m_i \times m_j$ matrices, for $1 \leq i, j \leq k$. Then we have $D_{ii} J_i -\gamma J_i D_{ii} =1\quad \text{for} \quad i=1, 2, \ldots, k$. Fix $i$, let $D_{ii}=(d_{pq})_{m_i\times m_i}$. Note that along the diagonal line of the matrix $D_{ii}J_i-\gamma J_i D_{ii}$ should be all ones, i.e.,
\[
1=0-\gamma d_{21},\;
1=d_{21}-\gamma d_{32},\;
\ldots,\;
1=d_{m_i-1\, m_i-2}-\gamma d_{m_i\, m_i-1}, \;\text{and}\;
1=d_{m_i\, m_i-1}
.\]
It then follows that $1+\gamma+\gamma^2+\ldots+\gamma^{m_i-1}=0$. But $\gamma$ is an $l$th primitive root of unity. Thus $m_i$ must be an integer multiple of $l$. Hence, $n=\sum_i m_i$ must be greater than or equal to $l$. It then follows from (\ref{nonsingular}, i) that $C$ has to be the $l\times l$ Jordan block with zeroes on the diagonal line, i.e., $C=QAQ^{-1}=X$.

Now, it is sufficient to show that $D=QBQ^{-1}$ must have the form $Y_{\alpha'\text{s}}$. It follows from $BA-\gamma AB= 1$ that $DX-\gamma XD=1$. We will explore the problem element-wise. By (\ref{solution}, i), we have
\begin{equation*}
(DX)_{ij}=\left\{
  \begin{array}{ll}
    d_{i\,j-1}, & j>1 \\
    0, & j=1
    \end{array}
\right.
\quad \text{and} \quad
(\gamma XD)_{ij}=\left\{
  \begin{array}{ll}
    \gamma d_{i+1\, j} & i<l \\
    0, & i=l
    \end{array}
\right.
\end{equation*}
Since $i+1\neq 1$ and $j-1\neq l$, it is clear that $d_{1l}$ is a free variable.

Case 1, $i<j$. Fix $1\leq k \leq l-1$, the $k$th superdiagonal line of $DX-\gamma XD$ has entries
\[(DX-\gamma XD)_{i, i+k}=d_{i, i+k-1} - \gamma d_{i+1, i+k}\quad \text{for} \quad i=1, 2, \ldots, l-k.\] That is,
$d_{1, k-1}-\gamma d_{2\, k}=0,\; d_{2\, k}-\gamma d_{3, k+1}=0\; \ldots\; d_{l-k, l-1}-\gamma d_{l-k+1, l}=0$. Then, inductively, \[d_{i, i+k-1}=\gamma^{l-k} d_{l-k+1, l}\quad \text{for} \quad i=1, 2, \ldots, l-k,\] Therefore, on the $l$th column of $D$, we get free variables $d_{2\, l}, \ldots, d_{l-1, l}$.

Case 2, $i=j$. We have
\[ 1=(DX-\gamma XD)_{ii}=\left\{
  \begin{array}{ll}
    0-\gamma d_{2\, 1} & i=1 \\
    d_{i, i-1} - \gamma d_{i+1, i} & 1<i<l\\
    d_{l, l-1}-0 & i=l
    \end{array}
\right.\]
Thus, the entries on the subdiagonal line of $D$ are
\[ d_{21}=-\f{1}{\gamma}=\sum_{i=0}^{l-2} \gamma^i, \quad
d_{32}=\f{d_{21}-1}{\gamma}=\sum_{i=0}^{l-3} \gamma^i,\quad \ldots, \quad d_{l-1, l-2}=1+\gamma, \; \text{and}\; d_{l, l-1}=1
\]

Case 3, $i>j$. We have
\[ 0=(DX-\gamma XD)_{ij}=\left\{
  \begin{array}{ll}
    0-\gamma d_{i+1, 1} & j=l\; \text{and}\; j<i<l\\
    d_{l, j-1}  & i=l\; \text{and}\; 1<j<i\\
    d_{i, j-1}-\gamma d_{i+1, j} & i<l\; \text{and}\; 1<j<i
    \end{array}
\right.\]
It is not hard to see, inductively, that $d_{ij}=0$ for any $2<i\leq l$ and $1\leq j \leq i-2$. The proposition follows.
\end{proof}

\begin{lem}\label{reduce-singular}
Any solution $(X, Y_{\alpha'\text{s}})$ in \emph{(\ref{singular})} is equivalent to a solution $(X, Y_{\beta})$ in \emph{(\ref{solution})}. The equivalence classes of solutions in \emph{(\ref{singular})} are $\; [\,(X, Y_{\beta})\,] \; \text{for}\; \beta \in \K $.
\end{lem}

\begin{proof}
Suppose $(X, Y_{\beta})$ is a solution as in (\ref{solution}). Direct computation shows that the following uppertriangular matrix $P$ is a nonsingular matrix such that $P^{-1}XP=X$.
\[P = \left(%
\begin{array}{cccccc}
 1 & p_{l-1} & p_{l-2} &\ldots & p_2 & p_1  \\
 \ & 1 & p_{l-1} &\ldots & p_3 & p_2  \\
 \ & \ & \ &\ddots &   \\
 \ & \ & \ & \ & 1 & p_{l-1} \\
 \ & \ & \ & \ & \ & 1\\
\end{array}%
\right)_{l\times l}
\]
where $p_k$ are scalars in $\K$ for $k=1, \ldots, l-1$. Then the matrices $X$ and $P^{-1}Y_{\beta}P$ also form a solution. It is shown in the proof of (\ref{singular}) that if $(X, D)$ is a solution then $D$ must be one of the form $Y_{\alpha'\text{s}}$. Thus, $P^{-1}Y_{\beta}P=Y_{\alpha'\text{s}}$ for some $\alpha_1, \ldots, \alpha_l$. Consider the matrices $PY_{\alpha'\text{s}}$ and $Y_{\beta}P$. The second row of $Y_{\beta}P$ is
\[
\sum_{i=0}^{l-2} \gamma^{i} \cdot (1\quad p_{l-1}\quad \ldots \quad p_1).
\]
Set $\bar{Y}_{\alpha'\text{s}}$ be the lower right $(l-1)\times (l-1)$ block of the matrix $Y_{\alpha'\text{s}}$ and $R$ be the $(l-1)\times l$ matrix obtained by removing the last row of the matrix $X$. Then the second row of $PY_{\alpha'\text{s}}$ can be written as
\[
(0\quad 1\quad p_{l-1}\quad \ldots \quad p_2)\cdot Y_{\alpha'\text{s}}=\Big(\sum_{i=0}^{l-2}\gamma^{i}\quad 0\quad \ldots \quad 0\Big)+  (1\quad p_{l-1}\quad\ldots\quad p_2)\cdot \bar{Y}_{\alpha'\text{s}}R
\]
It then follows from $PY_{\alpha'\text{s}}=Y_{\beta}P$ that
\[
\sum_{i=0}^{l-2} \gamma^{i}\cdot (p_{l-1}\quad \ldots \quad p_1)=(1\quad p_{l-1}\quad \ldots \quad p_2)\cdot \bar{Y}_{\alpha'\text{s}}
\]
Note that $\bar{Y}_{\alpha'\text{s}}$ is the sum of the upper triangular matrix
\[M=\left(%
\begin{array}{ccccc}
 \gamma^{l-2}\alpha_l & \gamma^{l-3}\alpha_{l-1} &\ldots   &\gamma\alpha_3 &\alpha_2  \\
  0 & \gamma^{l-3}\alpha_{l} &\ldots   &\gamma\alpha_4 &\alpha_3  \\

   \vdots & \vdots &\ddots   & \vdots &\vdots \\

 0& 0 &\ldots  &\gamma \alpha_l& \alpha_{l-1} \\
 0 &0 &\ldots  & 0&  \alpha_l \\
\end{array}%
\right)
\]
and the $(l-1)\times (l-1)$ subdiagonal matrix $L$ with entries
$\sum_{i=0}^{l-3} \gamma^{i},\quad \ldots,\; 1+\gamma,\; 1$ along the diagonal line. Then we have
\[
\sum_{i=0}^{l-2} \gamma^{i} \cdot (1\quad p_{l-1}\quad \ldots \quad p_1) - (1\quad p_{l-1}\quad \ldots \quad p_2)\cdot L = (1\quad p_{l-1}\quad \ldots \quad p_2)\cdot M
\]
and so
\[
\Big(\gamma^{l-2}\cdot p_{l-1}\quad \gamma^{l-3}(1+\gamma )\cdot  p_{l-2}\quad \ldots\quad \gamma  \sum_{i=0}^{l-3} \gamma^{i}\cdot p_2\quad \sum_{i=0}^{l-2} \gamma^{i}\cdot p_1 \Big) = (1\; p_{l-1}\; \ldots \; p_2)\cdot M
\]
Therefore
\begin{equation}\tag{*}
\left\{\begin{array}{ll}
p_{l-1}=\alpha_l &\\
p_k =\Big(\sum_{i=0}^{l-1-k} \gamma^{i}\Big)^{-1} \cdot \Big(\alpha_{k+1} + \sum_{i=k+1}^{l-1} \alpha_{l+k+1-i}\, \cdot p_i\Big) & \text{for} \quad k=l-2, \ldots, 1
\end{array}
\right.
\end{equation}
Moreover, it follows from $(PY_{\alpha'\text{s}})_{1,l}=(Y_{\beta}P)_{1,l}$ that \begin{equation*}\tag{**}
\beta=\alpha_1 + \sum_{i=1}^{l-1} \alpha_{l+1-i}\cdot p_i
\end{equation*}
This shows that any solution $(X, Y_{\alpha'\text{s}})$ is equivalent to a solution $(X, Y_{\beta})$ for some $\beta$ in $\K$. The equivalence condition is given by a polynomial condition on $\alpha$'s and $\beta$ that can be obtained inductively from (*) and (**). Simply arguing by determinant, we have the equivalence classes are
$\; [\,(X, Y_{\beta})\,] \; \text{for}\; \beta \in \K $.
\end{proof}

The preceding lemma shows that, up to equivalence, the only irreducible solution with both determinants equal to zero is the solution $(X, Y_0)$.  This solution corresponds to the only $x$- and $y$-torsion simple module $L(0)$ over the quantum Weyl algebra $\A$ at the root $\gamma$. (cf. \cite[Example 4.1]{Jor}) Next, we provide the equivalence classification for irreducible solutions found in \cite{BLTPS}, which are corresponding to the $x$- and $y$-torsion-free simple $\A$-modules.

\begin{lem}\label{reduce-nonsingular}
Any solution $(X_{\lambda}, Y_{\lambda \, b's})$ in \emph{(\ref{nonsingular} iii)} is equivalent to a solution $(X_{\lambda}, Y_{\lambda \, \eta})$ where
\[Y_{\lambda \, \eta} =\left(%
\begin{array}{cccccc}
 \frac{1}{(1-\gamma)\gamma \lambda} & 1 &0 &\ldots   &0 &0  \\
  0 & \frac{1}{(1-\gamma)\gamma^2 \lambda} &1 &\ldots   &0 &0  \\
  \ & \ & \ &\ & \ &\  \\
   \vdots & \vdots  & \vdots &\ddots  &\vdots & \vdots \\
 \ & \ & \ &\ & \ &\ \\
 0& 0 & 0 &\ldots & \frac{1}{(1-\gamma)\gamma^{l-1} \lambda}& 1 \\
\eta &0 &0 & \ldots & 0 &   \frac{1}{(1-\gamma)\gamma^l \lambda} \\
\end{array}%
\right)
\]
for some $\eta \in \K^{\times}$. The equivalence classes of solutions in \emph{(\ref{nonsingular} iii)} are $[\,(X_{\lambda}, Y_{\lambda\, \eta})\,]$ where $\eta \in \K^{\times}$ and $\lambda \in \K^{\times}/\langle \gamma \rangle$.
\end{lem}

\begin{proof}
Let $(X_{\lambda}, Y_{\lambda \, b's})$ and $(X_{\lambda'}, Y_{\lambda' \, c's})$ be two irreducible solutions as in (\ref{nonsingular} iii). They are equivalent only if $\lambda'^l=\lambda^l$ and $\prod_{i=1}^l b_i = \prod_{i=1}^l c_i$. Suppose $\lambda'=\gamma^i \lambda$ for some $1\leq i \leq l$.
Then, by using an appropriate permutation matrix, we have $(X_{\gamma^i \lambda},\, Y_{\gamma^i \lambda\, b's})$ is equivalent to $(X_{\lambda},\, Y_{\lambda\, b's})$, where the entries $b$'s in $Y_{\gamma^i \lambda\, b's}$ and those in $Y_{\lambda\, b's }$ are the same, up to the corresponding permutation. It then remains to show that $(X_{\lambda},\, Y_{ \lambda\, b's})$ is equivalent to $(X_{\lambda},\, Y_{\lambda\, \eta})$, where $\eta=\prod_{i=1}^l b_i$. This can be done by using a diagonal matrix $P$ with entries
\[1, \; b_1, \; b_1 b_2,\; \ldots,\; b_1\cdots b_{l-1}\]
along the diagonal line.
\end{proof}

Combining Proposition \ref{singular}, Lemma \ref{reduce-singular}, (\ref{nonsingular} iii) and Lemma \ref{reduce-nonsingular}, we have

\begin{prop}\emph{(}\cite[Theorem 5.8]{DGO}\emph{)}\label{all}
Any irreducible solution $(A, B)$ to the equation $yx-\gamma xy=1$, in which $BA \neq AB$, is equivalent to either the solution $(X_{\lambda}, Y_{\lambda \, \eta})$ in \emph{(\ref{reduce-nonsingular})} or the solution $(X, Y_{\beta})$ in \emph{(\ref{solution})}.
\end{prop}


\end{document}